\documentclass[11pt,reqno]{amsart}
\usepackage{amsmath,amsthm,amssymb,enumerate,color}
\usepackage{hyperref}

\usepackage[margin=1.5in]{geometry}

\newtheorem{theorem}{Theorem}[section]
\newtheorem{lemma}[theorem]{Lemma}
\newtheorem{proposition}[theorem]{Proposition}

\theoremstyle{definition}
\newtheorem{assumption}{Assumption}[section]

\theoremstyle{remark}
\newtheorem{remark}{Remark}[section]
\allowdisplaybreaks
\makeatletter
\newcommand{\subalign}[1]{%
  \vcenter{%
    \Let@ \restore@math@cr \default@tag
    \baselineskip\fontdimen10 \scriptfont\tw@
    \advance\baselineskip\fontdimen12 \scriptfont\tw@
    \lineskip\thr@@\fontdimen8 \scriptfont\thr@@
    \lineskiplimit\lineskip
    \ialign{\hfil$\m@th\scriptstyle##$&$\m@th\scriptstyle{}##$\crcr
      #1\crcr
    }%
  }
}
\makeatother
\newcommand\bR{\mathbb{R}}

\newcommand\cE{\mathcal{E}}
\newcommand\cF{\mathcal{F}}

\newcommand\cH{\mathcal{H}}

\newcommand\cK{\mathcal{K}}

\newcommand\cP{\mathcal{P}}
\newcommand\cB{\mathcal{B}}
\newcommand\cT{\mathcal{T}}

\newcommand\cU{\mathcal{U}}
\newcommand\cV{\mathcal{V}}

\newcommand{\E}{\mathbb{E}}
\newcommand{\R}{\mathbb{R}}
\newcommand{\N}{\mathbb{N}}
\newcommand{\bone}{\mathbf{1}}
\newcommand{\bB}{\mathbf{B}}

\begin{document}
\title[Feynman-Kac formula for SPDEs]{A Feynman-Kac formula for stochastic Dirichlet problems}

\author[M. Gerencs\'er]{M\'at\'e Gerencs\'er}
\address[M. Gerencs\'er]{Institute of Science and Technology, Austria}
\email{mate.gerencser@ist.ac.at}

\author[I. Gy\"ongy]{Istv\'an Gy\"ongy}
\address[I. Gy\"ongy]{University of Edinburgh, United Kingdom}
\email{i.gyongy@ed.ac.uk}


\begin{abstract}
A representation formula for solutions of stochastic partial differential equations with Dirichlet boundary conditions is proved. The scope of our setting is wide enough to cover the general situation when the backward characteristics that appear in the usual formulation are not even defined in the It\^o sense.
\end{abstract}

\maketitle

\section{Introduction}
The goal of the article is to present a Feynman-Kac formula for the solutions of stochastic partial differential equations (SPDEs). For deterministic PDEs such a probabilistic interpretation of the solution proved to be a remarkably useful tool to prove results that are either not available or are rather more difficult to obtain by purely analytic methods. It is hence not an unreasonable hope that a representation formula can also help in the stochastic case to obtain further information about the solutions. To indicate why obtaining Feynman-Kac formulae for SPDEs is not straightforward, let us recall a simple deterministic case. Take the 1-dimensional stochastic differential equations, parametrized by $t$ and $x$,
\begin{equation}\label{eq:intro X}
dX_s^{t,x}=\sigma_s(X_s^{t,x})\, \hat d B_s\quad\text{for }s\in[0,t],\quad X_{t}^{t,x}=x,
\end{equation}
where $B$ is a standard Wiener process and $\hat d B_s$ is its backward It\^o differential. The solution $X$ - or rather its continuous modification in $s,t,x$ - is often referred to as the backward characteristic. Under some mild conditions on $\sigma$ and $\psi$, $u_t(x):=\E\psi(X_0^{t,x})$ satisfies the Cauchy problem
$$
\partial_t u_t(x)=\tfrac{1}{2}\sigma^2_t(x)\Delta u_t(x),\quad u_0(x)=\psi(x).
$$
Now if we start from an initial value problem for SPDEs, in general - and in particular for the important example of the Zakai equation - the coefficients will be random and adapted to a forward filtration. Since in \eqref{eq:intro X} the noise evolves in reversed time, it becomes an equation in which the direction of the randomness in the coefficients and that of the noise do not match: the interpretation of a solution of such an equation and the subsequent analysis needed to prove the validity the formula is problematic.

When the equation is given on the whole space, this difficulty can be overcome by an elegant argument through fully degenerate SPDEs, see \cite{K_Kindof}, and one obtains a representation in which the role of the backward flows are taken over by spatial inverses of forward flows. While this gives some idea how a representation should look like when the equation is considered with some boundary conditions, the argument itself breaks down: the Dirichlet problem for degenerate equations is ill-posed. Here we take a more pragmatic approach and `build up' the representation formula from situations where the coefficients are deterministic and one can make sense of the backward characteristics. We note that the case of deterministic coefficients in a simplified setting were considered previously in \cite{FlanSchaum}, and indeed the first step in our proof is quite similar to that in \cite{FlanSchaum}, whose method in turn is based on \cite{KR_Char}.

As an application of the formula, we get an estimate the `localization' error one makes when imposing artificial boundary conditions to problems that are originally given on the whole space. The reason why this is of interest is that often the particular model that one wants to study, and gets the equation from, has no natural boundary conditions but is expected to vanish at infinity. One then may think then that setting the value to be zero on the boundary of a large enough domain is a good approximation of the original problem, and this is what we confirm and make precise below. 

The article is structured as follows. We continue with introducing some notations, after which in Section \ref{sec:formulation} the necessary objects for the Feynman-Kac formula are introduced and in Theorem \ref{thm:rep} the representation formula is stated. In Section \ref{sec:Preliminaries}, we collect some auxiliary results, and in Section \ref{sec:proof} we give the proof of Theorem \ref{thm:rep}. Section \ref{sec:loc} contains the above mentioned application for the localization error.

\subsection*{Notations}
Given a $d$-dimensional stochastic differential equation (SDE),
\begin{equation}\label{eq:flows explanation}
dX^i_t=\alpha^i_t(X_t)\,dt+\sum_k\beta^{ik}_t(X_t)\,dw^k_t,\quad i=1,2,...,d
\end{equation}
driven by a (possibly infinite) sequence of Wiener processes, the corresponding stochastic flow on an interval $[0,T]$ is a continuous random field $(X_{s,t}(x))_{0\leq s\leq t\leq T,x\in\R^d}$ such that for all $s$ and $x$, the process $(X_{s,t}(x))_{ s\leq t\leq T}$ is a solution of the equation \eqref{eq:flows explanation} with initial condition $X_{s,s}(x)=x$, and that furthermore for all $0\leq s\leq t\leq v\leq T$ and $x\in\R^d$,
$$
X_{t,v}(X_{s,t}(x))=X_{s,v}(x).
$$
When emphasizing the direction of the equation, one may also refer to it as the \emph{forward} flow, distinguishing it from \emph{backward} flows, which are the analogous objects for equations involving backward It\^o differentials. The existence of stochastic flows is known in quite large generality, see \cite{Kunita_Book97}, \cite{Kunita_StFleur}. Moreover, also under quite general assumptions, the mappings $X_{s,t}$ are diffeomorphisms from $\R^d$ onto itself, and hence one can also talk about the \emph{inverse} flow $(X_{s,t}^{-1}(x))_{0\leq s\leq t\leq T,x\in\R^d}$. 

The derivative of a function $f$ on $\R^d$ with respect to $x^i$ is denoted by $D_i$. We denote by $C^0$ the space of continuous functions, and by $C^\alpha$ the space of H\"older continuous functions with exponent $\alpha\in(0,1)$. For $\alpha\in[1,\infty)$, the space $C^\alpha$ consists of functions $v$ such that $D_lv\in C^{\alpha-\lfloor \alpha\rfloor}$ for all multiindex $l$ with length at most $\lfloor\alpha\rfloor$. For $p\geq2$, $L_p$ denotes the usual Lebesque space of generalized functions integrable to the $p$-th power, and $W^m_p$ the Sobolev space of generalized functions from $L_p$ whose distributional partial derivatives up to order $m$ are also generalized functions from $L_p$. When talking about an infinite sequence of functions $g=(g^k)_{k\in\N}$ belonging to a function space $C^\alpha$ or $W^m_p$, we always understand $g\in C^{\alpha}(l_2)$ or $g\in W^m_p(l_2)$, respectively. For a probability space with a product measure $P\otimes \hat P$, the notation $\E^{\hat P}$ will be used for integrating out with respect to the measure $\hat P$. The symbol $\E$ denotes integrating out all the random elements, in particular, in the previous situation of a product probability measure, $\E\E^{\hat P}X=\E X$ for integrable random variables $X$. Unless it is indicated 
otherwise, the 
summation convention with respect to repeated indices is used throughout the paper. 

\section{Formulation and main result}\label{sec:formulation}
Let $D\subset\R^d$ be a bounded $C^2$-domain, $(\Omega,\cF,P)$ be a complete probability space, $(\cF_t)_{t\geq0}$ be a filtration, and $(w^k_t)_{t\geq0, k=1,2,\ldots}$ be a sequence of independent $(\cF_t)$-Wiener martingales. The filtration is assumed to satisfy 
the ``the usual conditions", i.e., $\cF_0$ contains every event of probability zero, and 
$\cF_t=\cap_{t<s}\cF_s$. The predictable $\sigma$-algebra on $[0,\infty)\times\Omega$ is denoted 
by $\cP$.

We consider the following initial-boundary value problem
\begin{equation}							\label{eq:main Dirichlet}
\left\{\begin{array}{ll}
        du=[Lu+f]\,dt+[M^ku+g^k]\,dw^k_t\quad & \text{on } [0,T]\times D,\\
        u=0, & \text{on } (0,T]\times\partial D,\\
        u_0=\psi,
        \end{array}\right.
\end{equation}
where the differential operators $L$ and $M$ are given by
\begin{align*}
L\varphi&=\tfrac{1}{2}(\sigma\sigma^*+\rho\rho^*)^{ij}D_iD_j\varphi+b^iD_i\varphi+c\varphi,\\
M^k\varphi&=\sigma^{ik}D_i\varphi+\mu^k\varphi,
\end{align*}
with coefficients $\rho,b,c,\sigma,\mu$, and initial and free data $\psi,f,g$, 
defined for $(t,\omega,x)$ from $[0,\infty)\times\Omega\times\R^d$, 
such that they vanish for $x\notin D^{1}:=\{x:d(x,D)\leq 1\}$.
They are subject to the following assumptions, for some $\alpha>0$.

\begin{assumption}							\label{assumption coercivity}
There exists a $\lambda>0$ such that for all $t$, $\omega$, and $x\in D^{1/2}$, 
$$
(\rho\rho^\ast)_t(x)\geq\lambda I
$$
in the sense of positive semidefinite matrices, where $I$ is the identity matrix, and $\rho^*$ is the transpose of $\rho$.
\end{assumption}

\begin{assumption}                     \label{assumption regularitycoeff}
The coefficients $\rho,\sigma,b,c,\mu$ are predictable functions with values in $C^{2+\alpha}(\R^{d\times d})$, $C^{2+\alpha}((l_2)^d)$, $C^{1+\alpha}(\R^d)$, $C^{\alpha}(\R)$, and $C^{1+\alpha}(l_2)$, respectively, bounded uniformly in $t$ and $\omega$ by a constant $K$.
\end{assumption}

\begin{assumption}                       \label{assumption regularitydata}
The initial value, 
$\psi$ is an $\cF_0$-measurable random variable 
with values in $C^\alpha$. The free data, $f$ and $g$, are predictable processes with values in $C^{\alpha}$ and $C^{1+\alpha}(l_2)$, respectively, such that 
$$
\E\Big(|\psi|^2_{C^\alpha}+\int_0^T |f_t|^2_{C^\alpha}+|g_t|^2_{C^{1+\alpha}}\,dt\Big)\leq K.
$$
\end{assumption}

The above assumptions are more than sufficient to get from the general solution theory of SPDEs on domains in \cite{Kim_Lp} that the problem \eqref{eq:main Dirichlet} admits a unique solution $u$ in the following sense: $u$ belongs to $L_2(\Omega,C([0,T],L_2(D)))\cap L_2([0,T]\times\Omega,\cP,H^1_0(D))$, the equality
\begin{align*}
(u_t,\varphi)=(\psi,\varphi)+&\int_0^t(a^{ij}D_j u_s,-D_i\varphi)+((b^i+D_ja^ij)D_iu+cu_s+f,\varphi)\,ds\\
&
+\int_0^t(\sigma^{ik}D_iu_s+\mu^k u_s,\varphi)\,dw^k_s
\end{align*}
holds for all $\varphi$ smooth and compactly supported function on $D$ almost surely for all $t\in[0,T]$, and $(u_t(x))_{t\in[0,T],x\in D}$ is a continuous random field. Here $(\cdot,\cdot)$ denotes the $L_2(\R^d)$ inner product.

To introduce the representation of the solution $u$, let $(\hat w_t^{r})_{t\in[0,T];r=1,\ldots,d}$ be the $d$-dimensional Wiener process on the standard Wiener space $(\hat\Omega,\hat\cF,\hat P)$, where $\hat \Omega=C([0,T],\R^{d})$, $\hat\cF=\cB(\hat\Omega)$, and $\hat P$ is the Wiener measure. The associated forward characteristics to the problem \eqref{eq:main Dirichlet} are given by the SDE,
\begin{equation}            						 \label{flow}
dY_t=\beta_t(Y_t)\,dt-\sigma_t^k(Y_t)\,dw^k_t
-{\rho}^r_t(Y_t)\,d\hat{w}^{r}_t,
\end{equation} 
on the completion of the probability space $(\Omega\times\hat\Omega,\cF\otimes\hat\cF,P\otimes\hat P)$ where for $t\in[0,T],\,y\in\bR^d$,
$$
\beta_t(y)=-b_t(y)+\sigma^{ik}_{t}(y)D_i\sigma^k_t(y)
+\rho^{ri}_t(y)D_i\rho^r_t(y)+\sigma^k_t(y)\mu^k_t(y), 
$$
and $\sigma^k$, $\rho^r$ stand for the column vectors $(\sigma^{1k},\ldots,\sigma^{dk})$, $(\rho^{1r},\ldots,\rho^{dr})$, respectively. We shall also use the notation $\bar P=P\otimes \hat P$. Taking the stochastic flow $(Y_{s,t}(y))_{0\leq s\leq t\leq T,y\in\R^d}$ defined by \eqref{flow}, one can define the random times, for $t\in[0,T]$, $x\in\R^d$
\begin{equation}\label{eq:gamma}
\gamma_{t,x}=\sup\{s\in[0,t]:(s,Y_{s,t}^{-1}(x))\notin(0,T]\times D\},
\end{equation}
that is, the exit time of the inverse characteristic starting from $t,x$. Note however, that $\gamma$ is \emph{not} a stopping time in general with respect to either of the forward or backward filtrations.

Finally, introduce the processes $\eta$ and $U$ by
\begin{align}                                                      
d\eta_t(y)=&\bar c_t(Y_{0,t}(y))\eta_t(y)\,dt
+\mu^k_t(Y_{0,t}(y))\eta_t(y)\,dw^k_t, 
\quad\eta_0(y)=1,
\label{eq:rho}
\\  
dU_t(y)=&\{\bar c_t(Y_{0,t}(y))U_t(y)+\bar f_t(Y_{0,t}(y))\}\,dt
\nonumber
\\
&\quad+\{\mu^k_t(Y_{0,t}(y))U_t(y)+g^k_t(Y_{0,t}(y))\}\,dw^k_t,
\quad\quad U_0(y)=0\label{eq:U}
\end{align}
where 
$$
\bar c_t(x):=c_t(x)-\sigma^{ki}_t(x)D_i\mu^k_t(x), 
\quad 
\bar f_t(x)=f_t(x)-\sigma^{ki}_t(x)D_ig^k_t. 
$$
It is straightforward to check by Kolmogorov's criterion that $(\eta_t(y))_{t\in[0,T],y\in\R^d}$ and $(U_t(y))_{t\in[0,T],y\in\R^d}$ have continuous versions, for which we use the same notation. The `right-hand-side' of the Feynman-Kac formula will then read as
\begin{equation}\label{eq:v}
v_t(x):=\E^{\hat P}\left((\psi\eta_t)(Y_{0,t}^{-1}(x))\mathbf{1}_{\gamma_{t,x}=0}+(U_t-U_{\gamma_{t,x}}\tfrac{\eta_t}{\eta_{\gamma_{t,x}}})(Y_{0,t}^{-1}(x))\right).
\end{equation}
\begin{remark}
Note that integrating out the $\hat\omega$ variable gives (a version of) the conditional expectation given $\cF$. Using then the explicit expressions for $\eta$ and $U$, the formula can be written in the more familiar form 
\begin{align*}
&\E\Big[\psi(y)e^{\varphi_t(y)}\mathbf{1}_{\tau=0}
+\Big(\int_\tau^t\bar f_s(Y_{0,s}(y))e^{-\varphi_s(y)}\,ds
\\
&\quad\quad+\int_\tau^tg_s^k(Y_{0,s}(y))e^{-\varphi_s(y)}\,dw^k_s\Big)e^{\varphi_t(y)}\Big|_{\subalign{&\tau=\gamma_{t,x} \\ &y=Y_{0,t}^{-1}(x)}}\Big|\cF\Big]
\end{align*}
where  $\varphi_t(y)=\int_0^t(\bar c_s-(1/2)|\mu_s|^2)(Y_{0,s}(y))\,ds+\int_0^t\mu_s^k(Y_{0,s}(y))\,dw_s^k$.
\end{remark}

Fubini's theorem tells us that \eqref{eq:v} is meaningful for $dt\otimes dx\otimes dP$-almost every $t,x,\omega$, in particular, there is an event of full probability on which $v_t(x)$ is well defined for almost all $t,x$. To talk about $v$ as a random field however, we need a slightly better property, given by the following proposition.
\begin{proposition}\label{prop:field}
Under Assumptions \ref{assumption regularitycoeff}-\ref{assumption regularitydata} , there exists an event of full probability on which the right-hand-side of \eqref{eq:v} exists for all $t,x$, and it is jointly measurable in $\omega,t,x$.
\end{proposition}
The proof of this is given in Section \ref{sec:Preliminaries}. We are now in a position to state the main result.
\begin{theorem}								\label{thm:rep}
Under Assumptions \ref{assumption coercivity}-\ref{assumption regularitydata}, $u_t(x)=v_t(x)$ for all $t$, $dx\otimes dP$-almost everywhere.
\end{theorem}

\section{Preliminaries}\label{sec:Preliminaries}
\emph{Proof of Proposition \ref{prop:field}}.
Consider the random fields
\begin{align*}
\cU^{(n,m)}_t(x)&=(\psi\eta_t)(Y_{0,t}^{-1}(x))\mathbf{1}^{(m)}(\gamma^{(n)}_{t,x})+(U_t-U_{\gamma^{(n)}_{t,x}}\tfrac{\eta_t}{\eta_{\gamma^{(n)}_{t,x}}})(Y_{0,t}^{-1}(x)),
\\
\cU^{(m)}_t(x)&=(\psi\eta_t)(Y_{0,t}^{-1}(x))\mathbf{1}^{(m)}(\gamma_{t,x})+(U_t-U_{\gamma_{t,x}}\tfrac{\eta_t}{\eta_{\gamma_{t,x}}})(Y_{0,t}^{-1}(x)),
\\
\cU_t(x)&=(\psi\eta_t)(Y_{0,t}^{-1}(x))\mathbf{1}_{\gamma_{t,x}=0}+(U_t-U_{\gamma_{t,x}}\tfrac{\eta_t}{\eta_{\gamma_{t,x}}})(Y_{0,t}^{-1}(x)),
\end{align*}
for $n,m\in\mathbb{N}$, where 
$$
\mathbf{1}^{(m)}(x):=\mathbf{1}_{x<0}+(1-mx)\mathbf{1}_{x\in[0,1/m]},
$$
$$
\gamma^{(n)}_{t,x}:=n\int_0^{1/n}\gamma_{t,x}(\delta)\,d\delta,
$$
with the notation $D_\delta=\{x\in D:d(x,\partial D)>\delta\}$ 
and 
$$
\gamma_{t,x}(\delta)=\sup\{s\in[0,t]:(s,Y_{s,t}^{-1}(x))\notin(0,T]\times D_{\delta}\}.
$$
Note that $\gamma^{(n)}$ is continuous on $[0,T]\times D_{2/n}$: there exists a $\delta_0=\delta_0(\omega,n)$ such that if $|(t,x)-(t',x')|\leq\delta_0$, $t'>t$, and $x,x'\in D_{2/n}$ then $\inf_{s\in[t,t']}d(Y_{s,t'}^{-1}(x'),D_{1/n})<1/2n$ and $\sup_{s\in[0, t]}$ $|Y_{s,t}^{-1}(x)-Y_{s,t'}^{-1}(x)|\leq\varepsilon$. Therefore, we can write, for $\varepsilon\leq\delta\leq 1/n,$
$$
\gamma_{t',x'}(\delta-\varepsilon)\leq\gamma_{t,x}(\delta)\leq\gamma_{t',x'}(\delta+\varepsilon)\quad\text{ and hence }\quad\gamma_{t',x'}^{(n)}-\varepsilon n T\leq\gamma_{t,x}^{(n)}\leq\gamma_{t',x'}^{(n)}+\varepsilon n T.
$$

By Fubini's theorem there is an event $\tilde{\Omega}$ of full probability on which for almost all $t,x$, $\cU^{(n,m)}_t(x)$ is measurable as a function of $\hat \omega$. Since $\cU^{(n,m)}$ is continuous, this actually holds for all $t,x$. Since $\gamma_{t,x}(\delta)$ is right-continuous in $\delta$, the functions $\gamma_{t,x}^{(n)}$ converge to $\gamma_{t,x}$, and so $\cU^{(n,m)}_t(x)$ converge to $\cU^{(m)}_t(x)$ for all $t,x$. In particular, for $\omega\in\tilde{\Omega}$, $\cU^{(m)}_t(x)(\omega,\hat\omega)$ is a measurable function of $\hat\omega$ for all $t,x$. Taking then the $m\rightarrow\infty$ limit, this  holds for $\cU$ as well. Therefore $\cU_t(x)$ is a measurable function that is dominated by 
$$
\sup_{(t,y)\in[0,T]\times D^1}|\psi\eta_t|(y)+2\sup_{(s,t,y)\in[0,T]^2\times D^1}|U_t\tfrac{\eta_s}{\eta_{t}}|(y),
$$
which is integrable in $\hat\omega$ for almost all $\omega$, and therefore so is $\cU_t(x)$.
\qed

The following limit theorem is known, see e.g \cite{Kunita_StFleur}, \cite{LM_Flows}.
\begin{lemma}\label{lem:inverse flow convergence}
Let $\rho^n$, $\sigma^n$, $\mu^n$, and $b^n$ be coefficients satisfying Assumption \ref{assumption regularitycoeff} for $n=0,1,\ldots$ such that
$$
|\rho^n-\rho^0|_{C^{2+\alpha}(D^1)}+|\sigma^n-\sigma^0|_{C^{2+\alpha}(D^1)}
+|\mu^n-\mu^0|_{C^{1+\alpha}(D^1)}+|b^n-b^0|_{C^{1+\alpha}(D^1)}
$$
converges to $0$ in measure with respect to $dt\otimes dP$ as $n\rightarrow \infty$. Then
$$
\lim_{n\rightarrow\infty}\E\sup_{t\in[0,T]}|Y_{0,t}^{n}-Y_{0,t}^{0}|_{C^{1+\alpha}(D^1)}^2=0,
$$
$$
\lim_{n\rightarrow\infty}\E\sup_{t\in[0,T]}|Y_{0,t}^{n,-1}-Y_{0,t}^{0,-1}|_{C^{1+\alpha}(D^1)}^2=0.
$$
\end{lemma}

Define the set of trajectories that `touch' the boundary at some point as 
$$
\cT_D=\cup_{t>0}\cT_D(t),
$$
 where 
\begin{align*}
\cT_D(t)=\{f\in C([0,t],\R^d):\,\,
&\exists s\in[0,t], \varepsilon>0\,\,\text{such that}\\
  &f_s\in\partial D;\;\;\;\
  \forall r\in[(s-\varepsilon)\vee 0,(s+\varepsilon)\wedge t]\;\; f_r\in\bar D	\}.
\end{align*}

\begin{lemma}
One has, $dx\otimes dP\otimes d\hat P$-almost surely
$$
(Y_{0,s}(x))_{s\in[0,T]}\notin\cT_D.
$$
\end{lemma}
\begin{proof}
First notice that it suffices to prove the statement 
when one modifies the definition of $\cT_D$ to, say, $\cT'_D$, 
by changing $s\in[0,t]$ to $s\in[0,t)$ 
in the definition of $\cT_D(t)$. Indeed, 
the trajectory $(Y_{0,s}(x))_{s\in[0,T]}$ may only belong to 
$\cT_D\setminus\cT'_D$ if $Y_{0,T}(x)\in\partial D$, in other words, for
$$
(x,\omega,\hat\omega)\in\{(x,\omega,\hat\omega):x\in Y_{0,T}^{-1}(\partial D)(\omega,\hat\omega)\},
$$
and the latter set is of measure 0. The function
$$
d(x)=d(x,{\partial D})\text{ if }x\in \bar D, \quad d(x)=-d(x,\partial D)\text{ if } x\notin D
$$
is $C^2$ in a neighbourhood of $\partial D$, see e.g. \cite{GilbargT}. It is also easy to see that $|\nabla {d}_D|$ is separated away from zero in a neighbourhood of $\partial D$. One can then find a globally $C^2$ function $\hat d$ which agrees with $d$ on a neighbourhood of $\partial D$ and is separated away from zero outside that neighbourhood. Defining $Z_s:=\hat d(Y_{0,s}(x))$, we have
$$
(Y_{0,s}(x))_{s\in[0,T]}\notin\cT'_D\quad\Leftrightarrow\quad(Z_s)_{s\in[0,T]}\notin\cT'_{\R^+}.
$$
The process $Z$ has It\^o differential
$$
dZ_t=b_t\,dt+\sigma_t\,d\bar w_t
$$
with the Wiener process $\bar w=(w,\hat w)$ and with some bounded predictable functions $b$ and $\sigma$. Moreover, $d\langle Z\rangle_t\geq\lambda\mathbf{1}_{|Z_t|\leq\delta}\,dt$ for some positive constants $\lambda$ and $\delta$. Define the stopping times $\tau_0=0$ and for $i\geq 0$
\begin{align*}
\tau_{2i+1}&=\inf\{s\geq\tau_{2i}:\,|Z_s|\geq\delta\}\wedge T,\\
\tau_{2i+2}&=\inf\{s\geq\tau_{2i+2}:\,|Z_s|\leq\delta/2\}\wedge T.
\end{align*}
Note that the hitting times of $0$ of $Z$ can only occur on the intervals $[\tau_{2i},\tau_{2i+1}]$. Let us define $\bar b^i_t=\bar b_{(t+\tau_{2i})\wedge\tau_{2i+1}}$, $\bar\sigma^i_t=\bar \sigma_{(t+\tau_{2i})\wedge\tau_{2i+1}}$, and $\bar w^i_t=\bar w_{t+\tau_{2i}}$. Then for each $i\geq0$,
$$
Z^i_t:=\int_0^t\bar b^i_{s}\,ds+\int_0^t\bar\sigma^i_s\,d\bar w_s
$$
is a semimartingale with respect to the filtration $(\cF_{\tau_{2i}+s})_{s\geq0}$, satisfying $d\langle Z^i\rangle_t\geq\lambda\,dt$. Moreover, if $(Z_s)_{s\in[0,T]}\in\cT'_{\R^+}$, then $(Z_s^i)_{s\geq0}\in\cT_{\R^+ +a}$ for some $i\geq0$ and for one of $a=\delta/2$, $-\delta/2$, or $-Z_0$. Fixing $i$ and $a$ like so, to show that $(Z_s^i)_{s\geq0}\in\cT'_{\R^+ +a}$ has probability zero, we may change to an equivalent measure and hence by a Girsanov transform we may assume that $b=0$. Moreover, the probability also doesn't change if we perform a time change whose derivative is separated from $0$ and $\infty$, and so it actually suffices to see that $\bar P((B_s)_{s\geq0}\in\cT_{\R^++a})=0$ for a standard 1-dimensional Brownian motion. This is however known, and follows from
$$
\bar P((B_s)_{s\geq0}\in\cT_{\R^++a})\leq\sum_{r,q\in\mathbb{Q^+}}P(\min_{s\in[r,q]}B_s=a),
$$
and recalling that since the random variable $\min_{s\in[r,q]}B_s$ is absolutely continuous (in fact, with explicitly known density), and hence each term in the above sum is 0.
\end{proof}

\begin{proposition}\label{prop:touching}
For all $t\in[0,T]$, $dx\otimes P\otimes\hat P$-almost everywhere, 
\begin{equation}\label{eq:touching}
(Y_{s,t}^{-1}(x))_{s\in[0,t]}\notin\cT_D.
\end{equation}
\end{proposition}
\begin{proof}
By the previous lemma we can write
\begin{align*}
0&=\int_{D^1} \bar P((Y_{0,s}(y))_{s\in[0,t]}\in \cT_D)\,dy\\
&=\E\int_{D^1}\mathbf{1}_{(Y_{0,s}(y))_{s\in[0,t]}\in \cT_D}\,dy\\
&=\E\int_{D^1}\mathbf{1}_{(Y_{0,s}(Y_{0,t}^{-1}(x)))_{s\in[0,t]}\in \cT_D}|\det\nabla {Y_{0,t}^{-1}}(x)|\,dx.
\end{align*}
After interchanging the integral and expectation we conclude that for almost all $x$, 
$$
\E[\mathbf{1}_{(Y_{0,s}(Y_{0,t}^{-1}(x)))_{s\in[0,t]}\in \cT_D}|\det\nabla {Y_{0,t}^{-1}}(x)|]=0,
$$
and since, $\inf_{x\in D^1}|\det\nabla {Y_{0,t}^{-1}}(x)|$ is almost surely nonzero, the indicator is almost surely 0, which proves the claim.
\end{proof}

\begin{proposition}\label{prop:ui}
Let $\{f_i\}_{i\in I}$ be a uniformly integrable family of real-valued functions on a product of two measure spaces $(A,\mu)$ and $(B,\nu)$. Then $\{f_i(a,\cdot)\}_{i\in I}$ is uniformly integrable for almost all $a\in A$.
\end{proposition}

\begin{proof}
This is an easy consequence of de la Vall\'ee Poussin's theorem: we have a function $G$ such that $\lim_{t\rightarrow\infty}G(t)/t=\infty$ and
$$
\sup_{i\in I}\int G(f_i(a,b))\,d\mu(a)\, d\nu(b)<\infty.
$$
Then by Fubini's theorem, for almost all $a\in A$,
$$
\sup_{i\in I}\int G(f_i(a,b))\,d\nu(b)<\infty,
$$
which, by the converse direction of de la Vall\'ee Poussin's theorem, proves the claim.
\end{proof}

\section{Proof of Theorem \ref{thm:rep}}\label{sec:proof}
\emph{Step 1.} First consider the case when, further to the assumptions of the theorem, all coefficients and data are deterministic and do not depend on time. This was considered in \cite{FlanSchaum} with further assuming $f=g=0$ and $\psi|_{\partial D}=0$. The proof consists of two main steps: (a) establish a representation formula in terms of the appropriate backward flow (b) rewrite the formula in terms of the inverse flow, using the relationship between backward and inverse flows from \cite{Kunita_StFleur}. Part (a) follows very similarly to \cite{KR_Char} and \cite{FlanSchaum}, and is based on the Feynman-Kac formula for the \emph{deterministic} PDEs
\begin{equation}		\nonumber					
\left\{\begin{array}{ll}
        d\bar u=[L\bar u+f+q^k(M^k\bar u+g^k)]\,dt\quad & \text{on } [0,T]\times D,\\
        \bar u=0, & \text{on } (0,T]\times\partial D,\\
        \bar u_0=\psi,
        \end{array}\right.
\end{equation}
for arbitrary $q\in C^\infty([0,T], l_2)$. We therefore not give the details, but we note that for this representation it is not required that $\psi$ has 0 limit at the boundary, and hence neither is this assumption needed for Theorem \ref{thm:rep}. One obtains the formula through the backward characteristics
\begin{equation}\label{eq:X}
dX_t=(b_t(X_t)-\sigma^k_t\mu^k_t(X_t))\,dt+\sigma^k_t(X_t)\,\hat dw^k_t+\rho^r_t(X_t)\,\hat d\hat w_t^{r},
\end{equation}
where $\hat d$ denotes the backward It\^o differential, defined as in \cite{Kunita_StFleur}. Considering the corresponding backward flow $(X_{t,s}(x))_{0\leq s\leq t\leq T,x\in\R^d}$, the formula then reads as
\begin{align*}
u_t(x)&=\E^{\hat P}\Big[\psi(X_{t,0}(x))e^{\xi_0(x)}\mathbf{1}_{\tau_{t,x}=0}
+\int_t^{\tau_{t,x}} f_s(X_{t,s}(y))e^{\xi_s(x)}\,ds
\\
&\quad\quad+\int_t^{\tau_{t,x}}g_s^k(X_{t,s}(x))e^{\xi_s(x)}\,\hat dw^k_s\Big],
\end{align*}
for all $t$, $dx\otimes dP$-a.e., where
\begin{equation}\nonumber
\tau_{t,x}=\sup\{s\in[0,t]:(s,X_{t,s}^{-1}(x))\notin(0,T]\times D\}
\end{equation}
and $\xi_s(x)=\int_t^s(c_s-(1/2)|\mu_s|^2)(X_{t,s}(x))\,ds+\int_t^s\mu^k(X_{t,s}(x))\,\hat d w_s^k.$ 

For part (b) we give the full details here, partially because the transformation of the terms coming from the forcing is not trivial, and partially in order to correct a slight miscalculation in \cite{FlanSchaum} which in fact effects the formula therein itself (c.f. the definition \eqref{eq:rho}  of $\eta$ and (2.7) in \cite{FlanSchaum} of the corresponding term $\mu$). To this end, it is useful to introduce
$$
d\eta^B_t=c_t(X_t)\eta^B_t\,dt+\mu^k(X_t)\eta^B_t\,\hat dw^k_t
$$
$$
dU^B_t=f_t(X_t)\eta^B_t\,dt+g^k_t(X_t)\eta^B_t\,\hat dw^k_t
$$
and view $(Z^1,\ldots ,Z^{d+2}):=(X^1,X^2,\ldots,X^d,\eta^B,U^B)$ as a stochastic flow on $\R^{d+2}$. We can write, with the notation $\bar x=(x,1,0)$,
\begin{equation}\label{eq:repr backward flow}
u_t(x)=\E^{\hat P}[\psi((Z_{t,0}^1,\ldots,Z_{t,0}^d)(\bar x))Z_{t,0}^{d+1}(\bar x)\mathbf{1}_{\tau_{t,x}=0}+Z_{t,\tau_{t,x}}^{d+2}(\bar x)].
\end{equation}
We now invoke Theorem II.6.1. from \cite{Kunita_StFleur}. It states that $Z$ can be obtained as the inverse of the forward flow $V=(V^1,\ldots,V^{d+2})$, the coefficient of whose equation can be obtained from those of $Z$. Substituting in the formula, we get that 
$$V^j(x^1,\ldots,x^{d+2})=Y^j(x^1,\ldots,x^d)$$ 
for $j=1,\ldots,d$ (in particular, $\tau_{t,x}=\gamma_{t,x}$), and the equations for the last two coordinates read as
\begin{align}
dV^{d+1}_t&=(-c_t+|\mu_t|^2+\sigma_t^{ik}D_i\mu_t^k)(V_t^{d-})V^{d+1}_t\,dt
-\mu^k_t(V_t^{d-})V^{d+1}_t\,dw^k_t,\nonumber\\
dV^{d+2}_t&=(-f_t+\sigma^{ik}D_ig^k+g^k\mu^k)(V_t^{d-})V^{d+1}_t\,dt-g^k_t(V_t^{d-})V^{d+1}\,dw^k_t\nonumber
\end{align}
where $V^{d-}$ denotes the first $d$ coordinates of $V$. Let us also introduce the processes
\begin{align}
d\tilde \eta_t(x)&=(-c_t+|\mu_t|^2+\sigma_t^{ik}D_i\mu_t^k)(Y_{0,t}(x))\tilde \eta_t(x)\,dt
-\mu^k_t(Y_{0,t}(x))\tilde\eta_t(x)\,dw^k_t,  \nonumber\\
\tilde{\eta}_0(x)&=1
\label{eq:eta tilde}\\
d\tilde U_t(x)&=(-f_t+\sigma^{ik}D_ig^k+g^k\mu^k)(Y_{0,t}(x))\tilde\eta_t(x)\,dt
-g^k_t(Y_{0,t}(x))\tilde\eta_t(x)\,dw^k_t, \nonumber\\
\tilde{U}_0(x)&=0.
\label{eq:U tilde}
\end{align}
These processes look very similar to $V^{d+1}$, $V^{d+2}$, and indeed he relations between the two notions can be expressed as
\begin{align*}
V^{d+1}_{s,t}(y^1,\ldots,y^{d+2})&=y^{d+1}\tfrac{\tilde{\eta}_t}{\tilde{\eta}_s}(Y_{0,s}^{-1}(y^1,\ldots,y^d)),
\\
V^{d+2}_{s,t}(y^1,\ldots, y^{d+2})&=y^{d+2}+\tfrac{y^{d+1}}{\tilde\eta_s(Y_{0,s}^{-1}(y^1,\ldots,y^d)}(\tilde{U}_t-\tilde{U}_s)(Y_{0,s}^{-1}(y^1,\ldots,y^d)).
\end{align*}
Also note that simple applications of It\^o's formula yield that $\tilde\eta_{t}(x)=1/\eta_t(x)$ and $\tilde U_{t}(x)=-U_t(x)\tilde{\eta}_{t}(x)=-U_t(x)/\eta_t(x)$. Hence we can also write
\begin{align*}
V^{d+1}_{s,t}(y^1,\ldots,y^{d+2})&=y^{d+1}\tfrac{\eta_s}{\eta_t}(Y_{0,s}^{-1}(y^1,\ldots,y^d)),
\\
V^{d+2}_{s,t}(y^1,\ldots, y^{d+2})&=y^{d+2}-
y^{d+1}(\tfrac{\eta_s}{\eta_t}U_t-U_s)(Y_{0,s}^{-1}(y^1,\ldots,y^d)).
\end{align*}
Now when we write down the inverse of $V$ at the point $\bar x$, we can express the last two coordinates of the inverse in terms of $\eta$ and $U$:
\begin{align*}
((V_{s,t}^{-1})^1,\ldots(V_{s,t}^{-1})^d)(\bar x)&=((Y_{s,t}^{-1})^1,\ldots(Y_{s,t}^{-1})^d)(x^1,\ldots,x^d)
\\
(V_{s,t}^{-1})^{d+1}(\bar x)&=\tfrac{\eta_t}{\eta_s}(Y_{0,t}^{-1}(x^1,\ldots,x^d))
\\
(V_{s,t}^{-1})^{d+2}(\bar x)&=(V_{s,t}^{-1})^{d+1}(\bar x)(\tfrac{\eta_s}{\eta_t}U_t-U_s)(Y^{-1}_{0,t}(x^1,\ldots,x^d))
\\
&=(U_t-\tfrac{\eta_t}{\eta_s} U_s)(Y_{0,t}^{-1}(x^1,\ldots,x^d)).
\end{align*}
Hence substituting $V^{-1}$ in place of $Z$ in \eqref{eq:repr backward flow}, we recognize the right-hand-side as $v_t(x)$, and thus get the claim, for deterministic data and coefficients. 

\emph{Step 2.} One can then easily extend the formula to the case when all the coefficients and data are of the form
$$
a=\sum_{i=1}^na_i\mathbf{1}_{A_i}
$$
for some $n\geq1$, deterministic smooth functions $a_i$ of the spatial variable $x$, and $\cF_0$-measurable events $A_i$. The set of functions of this form will be denoted by $\cH(\cF_0)$.

The next case to consider is when $\psi\in\cH(\cF_0)$ and all other data and coefficient are of the form
$$
\bar a=\sum_{i=1}^n\bar a_i\mathbf{1}_{[t_{i-1},t_i)}
$$
for some $n\geq 1$, functions $\bar a_i\in\cH(\cF_{t_{i-1}})$, and times $0=t_0<t_1<\cdots t_n=T$. The set of functions of this form will be denoted by $\cH$. We demonstrate the argument for $n=2$, the generalization of which is straightforward. 
For $t\leq t_1$ we are in the previous situation, so we need only consider a fixed $t\in(t_1,T]$.
 The probability measure $\hat P$ on $\hat \Omega$ induces probability measures $\hat P^{(1)}$ and $\hat P^{(2)}$ on $\hat\Omega^{(1)}=C([0,t_1],\R^{2d})$ and $\hat\Omega^{(2)}=C([t_1,T],\R^{2d})$ by the mappings 
$$
(\hat w_t)_{t\in[0,T]}\mapsto(\hat w^{(1)}_t:=\hat w_t)_{t\in[0,t_1]},\quad
(\hat w_t)_{t\in[0,T]}\mapsto(\hat w^{(2)}_t:=\hat w_t-w_{t_1})_{t\in[t_1,T]},
$$
under which $\hat w^{(1)}$ and $\hat w^{(2)}$ are Wiener processes. We shall also use the notations $\gamma^{(i)}_{t,x}$, $\eta_t^{(i)}(x)$, and $U_t^{(i)}(x)$ for $i=1,2$,  that are defined similarly to $\gamma_{t,x}$, $\eta_t(x)$, and $U_t(x)$, but with `initial time' $t_{i-1}$ instead of $0$, `terminal time' $t_i$ instead of $T$.

By applying the formula in the already established cases, on one hand we get that $u_{t_1}=v_{t_1}$ holds $dx\otimes P$-almost everywhere, in other words, for an event $\tilde \Omega$ of full probability and $\omega\in\tilde\Omega$, $u_{t_1}$ and $v_{t_1}$ differ on a set $R(\omega)\subset\R^d$ of measure $0$.
 On the other hand we can write
\begin{align*}
u_t(x)&=
\E^{\hat P^{(2)}}\Big((u_{t_1}\eta^{(2)}_t)(Y_{t_1,t}^{-1}(x))\mathbf{1}_{\gamma^{(2)}_{t,x}=t_1}+(U^{(2)}_t-U^{(2)}_{\gamma^{(2)}_{t,x}}\tfrac{\eta^{(2)}_t}{\eta^{(2)}_{\gamma^{(2)}_{t,x}}})(Y_{t_1,t}^{-1}(x))\Big)
\end{align*}
$dx\otimes P$-almost everywhere. Clearly, $(u_{t_1}\eta^{(2)}_t)(Y_{t_1,t}^{-1}(x))$ and $(v_{t_1}\eta^{(2)}_t)(Y_{t_1,t}^{-1}(x))$ only differ by a finite random field $e(x)$ which may be nonzero only on 
$$
\{(\omega,\hat\omega,x):\omega\notin\tilde{\Omega}\text{ or }x\in Y_{t_1,t}(R(\omega))\}.
$$
Since $\sup_x|\nabla Y_{t_1,t}(x)|<\infty$ almost surely, this set has measure $0$, and therefore $\E^{\hat P^{(2)}}e(x)=0$, $dx\otimes P$-almost everywhere. Thus, we have
\begin{align}
u_t(x)&=
\E^{\hat P^{(2)}}\Big((v_{t_1}\eta^{(2)}_t)(Y_{t_1,t}^{-1}(x))\mathbf{1}_{\gamma^{(2)}_{t,x}=t_1}+(U^{(2)}_t-U^{(2)}_{\gamma^{(2)}_{t,x}}\tfrac{\eta^{(2)}_t}{\eta^{(2)}_{\gamma^{(2)}_{t,x}}})(Y_{t_1,t}^{-1}(x))\Big)\label{eq:rep0}
\end{align}
$dx\otimes P$-almost everywhere.

The concatenation mapping (that is, ``gluing'' $\hat w^{(1)}$ and $\hat w^{(2)}$ together) from $\hat \Omega^{(1)}\times\hat\Omega^{(2)}$ to $\hat \Omega$ maps the measure $\hat P^{(1)}\times \hat P^{(2)}$ to $\hat P$. Under this mapping 
\begin{enumerate}[(i)]
\item The flow $Y$ on $[0,t_0]$ driven by $\hat w^{(1)}$ and the one on $[t_0,T]$ driven by $\hat w^{(2)}$ also glue together to form a flow on $[0,T]$, driven by $\hat w$,
\item On $\{\gamma^{(2)}_{t,x}>t_1\}$, one has $\gamma_{t,x}=\gamma^{(2)}_{t,x}$, while on $\{\gamma^{(2)}_{t,x}=t_1\}$, one has $\gamma_{t,x}=\gamma^{(1)}_{t_1,y}|_{y=Y_{t_1,t}^{-1}(x)}$,
\item For $t\geq t_1$, one has $\eta_t(y)=\eta_{t_1}^{(1)}(y)\eta_t^{(2)}(Y_{0,t_1}(y))$,
\item For $t\geq t_1$, one has $U_t(y)=U_t^{(2)}(Y_{0,t_1}(y))+U_{t_1}^{(1)}(y)\tfrac{\eta_t(y)}{\eta_{t_1}(y)}$.
\end{enumerate} 
From these properties, along with of the flow identity $Y_{r,s}^{-1}(Y_{s,t}^{-1}(x))=Y_{r,t}^{-1}(x)$, the following identities follow easily:
\begin{align}
U_t^{(2)}(Y_{t_1,t}^{-1}(x))&=U_t(Y_{0,t}^{-1}(x))-U_{t_1}^{(1)}(Y_{0,t}^{-1}(x))\tfrac{\eta_t(Y_{0,t}^{-1}(x))}{\eta_{t_1}(Y_{0,t}^{-1}(x))},
\label{eq:rep1}
\\
U_t(Y_{0,t}^{-1}(x))&=U_t^{(2)}(Y_{t_1,t}^{-1}(x))+U^{(1)}_{t_1}(Y_{0,t}^{-1}(x))\eta_t^{(2)}(Y_{t_1,t}^{-1}(x)),\label{eq:rep4}
\\
\tfrac{\eta^{(2)}_t(Y_{t_1,t}^{-1}(x))}{\eta^{(2)}_{\gamma^{(2)}_{t,x}}(Y_{t_1,t}^{-1}(x))}
&=\tfrac{\eta_t(Y_{0,t}^{-1}(x))}{\eta_{\gamma^{(2)}_{t,x}}(Y_{0,t}^{-1}(x))},
\label{eq:rep2}
\\
\eta_t(Y_{0,t}^{-1}(x))&=(\eta_{t_1}^{(1)}
(Y_{0,t_1}^{-1}(\cdot))
\eta^{(2)}_t(\cdot))(Y_{t_1,t}^{-1}(x)).
\label{eq:rep3}
\end{align}
Therefore, substituting in \eqref{eq:rep0} the definition of $v_{t_1}$, we can write
\begin{align}
u_t(x)
&
=\E^{\hat P}\Big(
\mathbf{1}_{\{\gamma^{(2)}_{t,x}>t_1\}}\Big[
(U^{(2)}_t
-U^{(2)}_{\gamma^{(2)}_{t,x}}
\tfrac{\eta^{(2)}_t}{\eta^{(2)}_{\gamma^{(2)}_{t,x}}})(Y_{t_1,t}^{-1}(x))\Big]
\nonumber\\
&\quad\quad
+\mathbf{1}_{\{\gamma^{(2)}_{t,x}=t_1\}}\Big[
\{(\psi\eta_{t_1}^{(1)})
(Y_{0,t_1}^{-1}(\cdot)\mathbf{1}_{\gamma^{(1)}_{t_1,\cdot}=0})
\eta^{(2)}_t(\cdot)\}(Y_{t_1,t}^{-1}(x))
\nonumber\\
&\quad\quad
+\{(U^{(1)}_{t_1}-U^{(1)}_{\gamma^{(1)}_{t_1,\cdot}}
\tfrac{\eta^{(1)}_{t_1}}{\eta^{(1)}_{\gamma^{(1)}_{t_1,\cdot}}})
(Y_{0,t_1}^{-1}(\cdot))
\eta_t^{(2)}(\cdot)\}(Y_{t_1,t}^{-1}(x))
\nonumber\\
&\quad\quad
+(U^{(2)}_t-U^{(2)}_{t_1}
\tfrac{\eta^{(2)}_t}{\eta^{(2)}_{t_1}})
(Y_{t_1,t}^{-1}(x))\Big]\Big)
\nonumber\\
&
=\E^{\hat P}\Big(
\mathbf{1}_{\{\gamma^{(2)}_{t,x}>t_1\}}\Big[
(U_t
-U_{\gamma_{t,x}}
\tfrac{\eta_t}{\eta_{\gamma_{t,x}}})(Y_{0,t}^{-1}(x))\Big]
\nonumber\\
&\quad\quad
+\mathbf{1}_{\{\gamma^{(2)}_{t,x}=t_1\}}\Big[
(\psi\eta_t)(Y_{0,t}^{-1}(x))
\mathbf{1}_{\gamma_{t,x}=0}
\nonumber\\
&\quad\quad
+U^{(1)}_{t_1}(Y_{0,t}^{-1}(x))\eta_t^{(2)}(Y_{t_1,t}^{-1}(x))
-(U_{\gamma_{t,x}}^{(1)}
\tfrac{\eta_{t}}{\eta^{(1)}_{\gamma_{t,x}}})(Y_{0,t}^{-1}(x))
\nonumber\\
&\quad\quad
+U_t^{(2)}(Y_{t_1,t}^{-1}(x))\Big]\Big)\label{eq:proof0}
\end{align}
$dx\otimes P$-almost everywhere. Indeed, in transforming the first line we used \eqref{eq:rep1} twice as well as \eqref{eq:rep2}, in the second we used \eqref{eq:rep3}, in the third we used \eqref{eq:rep3} again, and the fourth line was not changed, since $U^{(2)}_{t_1}=0$. Making then use of \eqref{eq:rep4} and of the fact that on $\{\gamma^{(2)}_{t,x}=t_1\}$ one has
$$
U_{\gamma_{t,x}}^{(1)}=U_{\gamma_{t,x}},
\quad
\eta^{(1)}_{\gamma_{t,x}}=\eta_{\gamma_{t,x}},
$$
we can write
\begin{align*}
u_t(x)&
=\E^{\hat P}\Big(
\mathbf{1}_{\{\gamma^{(2)}_{t,x}>t_1\}}\Big[
(U_t
-U_{\gamma_{t,x}}
\tfrac{\eta_t}{\eta_{\gamma_{t,x}}})(Y_{0,t}^{-1}(x))\Big]
\nonumber\\
&\quad\quad
+\mathbf{1}_{\{\gamma^{(2)}_{t,x}=t_1\}}\Big[
(\psi\eta_t)(Y_{0,t}^{-1}(x))
\mathbf{1}_{\gamma_{t,x}=0}
\nonumber\\
&\quad\quad
+U_t(Y_{0,t}^{-1}(x))
-(U_{\gamma_{t,x}}
\tfrac{\eta_{t}}{\eta_{\gamma_{t,x}}})(Y_{0,t}^{-1}(x))
\Big]\Big)=v_t(x)\nonumber
\end{align*}
as claimed. The proof of the formula for data and coefficients from the class $\cH$ is hence finished.

\emph{Step 3.} For the general case, take coefficients and data $\rho^n$, $\sigma^n$, $\mu^n$, $b^n$, $c^n$, $\psi^n$, $f^n$, and $g^n$ of class $\cH$ such that they satisfy Assumptions \ref{assumption regularitycoeff}-\ref{assumption regularitydata}, $|c-c^n|_{C^\alpha}\rightarrow 0$ in measure with respect to $dt\otimes dP$, 
$$
|\psi-\psi^n|_{C^\alpha}+\int_0^T |f_t-f_t^n|^2_{C^\alpha}+|g_t-g_t^n|^2_{C^{1+\alpha}}\,dt\rightarrow0
$$
in probability, and the remaining coefficients converge as in the condition of Lemma \ref{lem:inverse flow convergence}. The existence of such approximation is well-known and follows from standard arguments. From the previous parts we can write
\begin{equation}\label{eq:approx rep}
u^n_t(x)=\E^{\hat P}\Big((\psi^n\eta^n_t)(Y_{0,t}^{n,-1}(x))\mathbf{1}_{\gamma^n_{t,x}=0}+(U^n_t-U^n_{\gamma^n_{t,x}}\tfrac{\eta^n_t}{\eta^n_{\gamma^n_{t,x}}})(Y_{0,t}^{n,-1}(x))\Big).
\end{equation}
$dx\otimes dP$-almost everywhere for every $n$, where $\gamma^n$, $\eta^n$, and $U^n$ are defined analogously to \eqref{eq:gamma}, \eqref{eq:rho}, and \eqref{eq:U}. The left-hand-side of \eqref{eq:approx rep} converges to $u_t(x)$ almost surely for each $(t,x)\in[0,T]\times D$, by the theory of SPDEs of domains, see \cite{K_W2m}, \cite{Kim_Lp}. For the convergence of the right-hand-side, first note that by Proposition \ref{prop:touching}, we may replace it by
$$
\E^{\hat P}\mathbf{1}_{(Y_{s,t}^{-1}(x))_{s\in[0,t]}\notin\cT_D}\Big((\psi^n\eta^n_t)(Y_{0,t}^{n,-1}(x))\mathbf{1}_{\gamma^n_{t,x}=0}+(U^n_t-U^n_{\gamma^n_{t,x}}\tfrac{\eta^n_t}{\eta^n_{\gamma^n_{t,x}}})(Y_{0,t}^{n,-1}(x))\Big).
$$
By Vitali's convergence theorem it suffices to prove that for all $t$, $dx\otimes dP$-almost everywhere, the quantity under the sign $\E^{\hat P}$ 
\begin{enumerate}[(i)]
\item converges $\hat P$-a.s.
\item is uniformly integrable in $\hat \omega$.
\end{enumerate}
Moreover, recalling also Proposition \ref{prop:ui}, instead of (ii) it actually suffices to prove that the family 
$$
\sup_{(t,y)\in[0,T]\times D^1}|\psi^n\eta^n_t|(y)+2\sup_{(s,t,y)\in[0,T]^2\times D^1}|U^n_t\tfrac{\eta^n_s}{\eta^n_{t}}|(y),
$$
is uniformly integrable in $(\omega,\hat\omega)$. Since we have uniform (in $n$) bounds on the coefficients of the SDEs \eqref{eq:rho}, \eqref{eq:U}, the uniform integrability follows from standard moment bounds, see e.g. \cite{K_Controlled}

Concerning (i), from Lemma \ref{lem:inverse flow convergence} we have that that the inverse flow trajectories $(Y_{s,t}^{n,-1}(x))_{s\in[0,t]}$ converge to $(Y_{s,t}^{-1}(x))_{s\in[0,t]}$ in the supremum norm. If furthermore $(Y_{s,t}^{-1}(x))_{s\in[0,t]}\notin\cT_D$, then $\gamma_{t,x}^n$ also converges to $\gamma_{t,x}$ and $\mathbf{1}_{\gamma^n_{t,x}=0}$ to $\mathbf{1}_{\gamma_{t,x}=0}$. For the convergence of the other terms it suffices to see that $\eta^n$, $1/\eta^n$, and $U^n$ converge along a subsequence uniformly in space and time, and hence when substituting in the space-time parameters convergent quantities, in our case $Y_{0,t}^{n,-1}(x)$ and $\gamma^n_{t,x}$, the resulting quantity also converges. The proof for the uniform convergence is virtually identical for $\eta^n$, $1/\eta^n$, and $U^n$, so we only detail the first. Let $p>1/\alpha$,  $\Lambda_n$ denote a $1/n^{p}$-net of $[0,T]\times D$ and $\Pi_n$ a function $[0,T]\times D\rightarrow \Lambda_n$ such that $|\Pi_n(t,x)-(t,x)|\leq 1/n^p$ for all $(t,x)\in [0,T]\times D$. Since $\eta^n_t(x)$ converges in $L_1(\Omega)$ to $\eta_t(x)$ for all $(t,x)$, we can find a subsequence $(k(n))_{n\in\N}$ such that
$$
\sum_{(t,x)\in\Lambda_n}\E|\eta_t^{k(n)}(x)-\eta_t(x)|\leq n^{-3}.
$$
Therefore, by Markov's inequality
$$
\bar P(A_n^m):=\bar P(\max_{(t,x)\in\Lambda_n}|\eta_t^{k(n)}(x)-\eta_t(x)|\geq m n^{-1})\leq n^{-2}m^{-1}.
$$
Also, we have $\E(|\eta^{k(n)}|_{C^{\alpha/2}}+|\eta|_{C^{\alpha/2}})^2\leq N$ for all $n$, with some constant $N=N(d,\alpha,K, T, D)$. Applying Markov's inequality again, we have, for any $\delta>0$,
$$
\bar P(B_n^m):=\bar P(|\eta^{k(n)}|_{C^{\alpha/2}}+|\eta|_{C^{\alpha/2}}\geq mn^{1/2+\delta})\leq Nn^{-1-2\delta}m^{-2}.
$$
For each $m$, we can therefore write on $C^m:=\cap_{n\in\N}((A_n^m)^c\cap (B_n^m)^c)$, for all $n\in\N$ and for any $(t,x)$,
\begin{align}
|\eta^{k(n)}_t(x)-\eta_t(x)|&\leq|\eta^{k(n)}(\Pi_n (t,x))-\eta(\Pi_n (t,x))|+(n^{-p})^{\alpha/2}(|\eta^{k(n)}|_{C^{\alpha/2}}+|\eta|_{C^{\alpha/2}})  
\nonumber\\
&\leq mn^{-1} +mn^{-p\alpha/2+1/2+\delta}.\nonumber
\end{align}
If $\delta<(p\alpha-1)/2$, which we can achieve, then the right-hand side goes to $0$, uniformly in $t$ and $x$. It remains to notice that $\bar P(C^m)\geq 1-N'm^{-1}$ with some constant $N'=N'(N,\delta)$ and therefore the uniform convergence holds on the set $\cup_{m\in\N}C^m$ of full probability. In other words, the set of $\omega$-s where the uniform convergence holds $\hat P$-almost surely, has probability $1$, which finishes the proof.

\qed

\begin{remark}
As it is seen from the proof, one could also write the formula in terms of $\tilde \eta$ and $\tilde U$, as defined in \eqref{eq:eta tilde}-\eqref{eq:U tilde}. In fact, from the inversion of the flows this would be somewhat more natural, but the formula as written is more consistent with the existing literature, e.g. \cite{FlanSchaum}, \cite{K_Kindof}, \cite{LM_Flows}.
\end{remark}

\section{Localization errors for artificial boundary conditions}\label{sec:loc}
Let us turn to an application of the formula. In this section we consider equations on the whole space 
\begin{equation}							\label{eq:loc R^d}
du=[Lu+f]\,dt+[M^ku+g^k]\,dw^k_t\quad  \text{on } [0,T]\times \R^d,\quad u_0=\psi.
\end{equation}
We are interested how close to $u$ is the solution of the truncated problem
\begin{equation}							\label{eq:loc Dirichlet}
\left\{\begin{array}{ll}
        du^R=[Lu^R+f]\,dt+[M^ku^R+g^k]\,dw^k_t\quad & \text{on } [0,T]\times B_R,\\
        u^R=0, & \text{on } (0,T]\times\partial B_R,\\
        u^R_0=\psi.
        \end{array}\right.
\end{equation}
The differential operators $L$ and $M$ have the same form as in \eqref{eq:main Dirichlet}, and while our assumptions are similar to Assumptions \ref{assumption coercivity}-\ref{assumption regularitydata}, due to some differences and for the convenience of the reader we state them separately. 

\begin{assumption}							\label{assumption loc coercivity}
There exists a $\lambda>0$ such that for all $t$, $\omega$, and $x$, 
$$
(\rho\rho^\ast)_t(x)\geq\lambda I
$$
in the sense of positive semidefinite matrices, where $I$ is the identity matrix.
\end{assumption}

\begin{assumption}                     \label{assumption loc regularitycoeff}
The coefficients $\rho,\sigma,b,c,\mu$ are predictable functions with values in $C^{2}(\R^{d\times d})$, $C^{2}((l_2)^d)$, $C^{1}(\R^d)$, $C^{1}(\R)$, and $C^{2}(l_2)$, respectively, bounded uniformly in $t$ and $\omega$ by a constant $K$.
\end{assumption}

\begin{assumption}                       \label{assumption loc regularitydata}
For some $p>d$, the initial value, 
$\psi$ is an $\cF_0$-measurable random variable 
with values in $W^1_p$. The processes $f$ and $g$ are predictable with values in $W^1_p$ and $W^2_p$, respectively, such that
$$
\cK_{1,p}(\psi,f,g):=|\psi|_{W^1_p}+\|f\|_{L_{p}([0,T],W^1_p)}+\|g\|_{L_{p}([0,T],W^{2}_p)}<\infty
$$
almost surely.
\end{assumption}

Introduce the shorthand $\bB_R=[0,T]\times B_R$. The result on localization of linear equations reads as follows.

\begin{theorem}\label{thm:localization}
Let Assumptions \ref{assumption loc coercivity}-\ref{assumption loc regularitydata} hold. Then for any $R>1$, $q>1$, $\varepsilon\in(0,1]$, and $\nu\in(0,1)$, one has 
\begin{equation}\label{eq:loc estimate linear}
\E\|u-u^R\|_{L_\infty(\bB_{R-\nu R^\varepsilon})}\leq Ne^{-\delta R^{2\varepsilon}}\E^{1/q}\cK^q_{1,p}(\psi,f,g),
\end{equation}
where the constants $N$, $\delta>0$ depend on $p$, $q$, $\varepsilon$, $\nu$, $\lambda$, $d$, $K$, $T$.
\end{theorem}

\begin{remark}
It should be noted that in the generality considered here, for the localized equation \eqref{eq:loc Dirichlet} there are not known approximating schemes with optimal rate, and hence it is likely preferable to use the localization of \cite{GG_Loc}. Therein, even though all data have compact support, the localized equation still can be considered on the whole space, and be approximated as such (see e.g. the full discretization scheme in \cite{GG_Loc}). One advantage of the method presented here is that coercivity is preserved, in fact, the equation itself does not change at all. Therefore, if a specific equation has efficient schemes on domains (which usually do strongly rely on coercivity), then this type of localization can be favourable. 
\end{remark}

\begin{remark}
We also note that while the extension of the above error estimate to nonlinear equation is not an easy task in this generality, Theorem \ref{thm:localization} still can be a useful tool in nonlinear situations. For example, take some sufficiently nice functions $\bar f$ and $\bar g$ mapping from $\R$ to $\R$, let $u$ be the solution of \eqref{eq:loc R^d} with $f$ and $g$ replaced by the semilinear terms $\bar f(u)$ and $\bar g(u)$, and similarly change the equation \eqref{eq:loc Dirichlet} for $u^R$. If one then defines $\tilde u^R$ as the solution of \eqref{eq:loc Dirichlet} with $f$ and $g$ replaced by $\bar f(u)$ and $\bar g(u)$, respectively, then Theorem \ref{thm:localization} gives a bound for $u-\tilde u^R$. It then remains to estimate $\tilde u^R-u^R$, which is perhaps a challenging task in general, but under some additional assumptions on the operators $L$ and $M$ - which, as mentioned above, are necessary anyway to be able to approximate the localized problem - it may not be insurmountable. This direction is left for future work.
\end{remark}

Before turning to the proof, let us recall some estimates from a Sobolev space theory of degenerate equations in \cite{GGK14}: under Assumptions \ref{assumption loc regularitycoeff}-\ref{assumption loc regularitydata}, one has for all $q\in(0,\infty)$,
\begin{equation}\label{eq:seged degen}
\E\|u\|_{L_{\infty}([0,T],W^1_p)}^q\leq N\E\cK^q_{1,p}(\psi,f,g),
\end{equation}
where $N$ depends only on $p$, $q$, $\lambda$, $d$, and $K$.
We also invoke a probability estimate for the flows from \cite{GG_Loc}. While in fact in \cite{GG_Loc}, this is only proved for $\varepsilon=1$, this slight generalization is straightforward.
\begin{lemma}\label{lem: exit prob}
Let $Y$ be as in \eqref{flow} and define the event
$$
H_R:=\Big[\sup_{(t,x)\in\bB_{R-\nu R^\varepsilon}}|Y_{0,t}^{-1}(x)|>R-(\nu/2)R^\varepsilon\Big]\cup\Big[\sup_{(t,x)\in\bB_{R-(\nu/2)R^\varepsilon}}|Y_{0,t}(x)|>R\Big].
$$
Then
$$
\bar P(H_R)\leq Ne^{-\delta R^{2\varepsilon}},
$$
where $N$ and $\delta>0$ depend only on $\lambda$, $d$, $K$, $T$, $\nu$, and $\varepsilon$. 
\end{lemma}

\emph{Proof of Theorem \ref{thm:localization}}
We may and will assume that the coefficients are smooth enough so that Assumption \ref{assumption regularitycoeff} is satisfied. Indeed, if the estimate is obtained for such smoothed coefficients, the passage to the limit is justified by \cite{K_Lp} (for $u$) and 
by \cite{Kim_Lp} (for $u^R$). The constant $N$ may change from line to line, but always has the dependence specified in the theorem.

Let $Y$ be as above, $\eta$ and $U$ as in \eqref{eq:rho}-\eqref{eq:U}, 
and $\gamma^R$ as in \eqref{eq:gamma}, with $D=B_R$. By Theorem \ref{thm:rep} 
and recalling the representation on the whole space by \cite{K_Kindof} we have that
\begin{align*}
u_t(x)&=\E^{\hat P}\Big((\psi\eta_t)(Y_{0,t}^{-1}(x))+{U}_t(Y_{0,t}^{-1}(x))\Big),\\
{u}^R_t(x)&=\E^{\hat P}\Big((\psi\eta_t)(Y_{0,t}^{-1}(x))\mathbf{1}_{\gamma^R_{t,x}=0}+( U_t-U_{\gamma_{t,x}} \tfrac{\eta_t}{\eta_{\gamma^R_{t,x}}})(Y_{0,t}^{-1}(x))\Big).
\end{align*}
Take a parameter $\bar p\geq 1$, with which we will eventually tend to infinity. 
Denote the quantities under the $\E^{\hat P}$ sign by $\cU_t(x)$ and $\cU_t^R(x)$, respectively, 
the norm in $L_\infty([0,T],L_{\bar p}(B_{R-\nu R^\varepsilon}))$ 
by $\|\cdot\|_{(R)}$, and note that on the complement of $H_R$, $\cU_t(x)=\cU_t^R(x)$ for all $(t,x)\in\bB_{R-\nu R^\varepsilon}$. By Minkowski and H\"older inequalities and Lemma \ref{lem: exit prob}
\begin{align}
\E\|u-{u}^R\|_{(R)}&=\E\|\E^{\hat P}(\cU-\cU^R)\|_{(R)}\nonumber\\
&\leq
\E\E^{\hat P}\|\cU-\cU^R\|_{(R)}\nonumber\\
&=
\E\bone_{H_R}\|\cU-\cU^R\|_{(R)}\nonumber\\
&\leq
(\bar P(H_R))^{1/q'}\E^{1/q}\|\cU-\cU^R\|_{(R)}^q\nonumber\\
&\leq
Ne^{-\delta R^{2\varepsilon}}
(\E^{1/q}\|\cU\|^q_{(R)}+
\E^{1/q}\|\cU^R\|^q_{(R)}),\label{eq:10}
\end{align}
where $q\in(1,\infty)$ and $q'=q/(q-1)$. 
At this stage we can make use of the fact, 
see again \cite{K_Kindof}, that $\cU$ is in fact 
a solution of the fully degenerate SPDE
\begin{equation}		\label{eq:00}			
d\cU=[L\cU+ f]\,dt+[M^k\cU+ g^k]\,dw^k_t
+\rho^{ir}D_i\cU\,d\hat w^r_t\quad  \text{on } [0,T]\times \R^d,\quad \cU(0)=\psi.
\end{equation}
By elementary inequalities, Sobolev's embedding, 
and \eqref{eq:seged degen}, we have
\begin{align} 
\E\|\cU\|^q_{(R)}&\leq N(2R)^{dq/\bar p}\E\|\cU\|^q_{L_\infty(\bB_R)} 
\nonumber\\
&\leq N(2R)^{dq/\bar p}\E\|\cU\|^q_{L_\infty([0,T]\times\R^d)}
\nonumber\\
& \leq N(2R)^{dq/\bar p}\E\|\cU\|^q_{L_\infty([0,T],W^1_p)}
\nonumber\\
&\leq N(2R)^{dq/\bar p}\E\cK_{1,p}(\psi,f,g)^q=:\cE^q.\label{eq:01}
\end{align}
As for $\cU^R$, let us write
$$
\E\|\cU^R\|^q_{(R)}
\leq
3^{q-1}
(\E\|\cV^1\|^q_{(R)}+\E\|\cV^2\|^q_{(R)}+\E\|\cV^3\cV^4\cV^5\|^q_{(R)}),
$$
where
\begin{align*}
\cV^1_t(x)&=(\psi\eta_t)(Y_{0,t}^{-1}(x)),
&\quad
\cV^2_t(x)&=U_t(Y_{0,t}^{-1}(x)),
\\
\cV^3_t(x)&=U_{\gamma^R_{t,x}}(Y_{0,t}^{-1}(x)),
&\quad
\cV^4_t(x)&=\eta_t(Y_{0,t}^{-1}(x)),
&\quad
\cV^5_t(x)&=\eta^{-1}_{\gamma^R_{t,x}}(Y_{0,t}^{-1}(x)).
\end{align*}
Applying again the representations on the whole space, 
we have that $\cV^1$, $\cV^2$, and $\cV^4$ 
are solutions of equations of type \eqref{eq:00}, 
with the data $(\psi, f, g)$ replaced by $(\psi,0,0)$, $(0, f,g)$, 
and $(1,0,0)$, respectively. 
Hence \eqref{eq:seged degen} yields estimates of type \eqref{eq:01} 
for $\cV^1$ and $\cV^2$. 
 One can also verify by direct calculation (see e.g. \cite{GG_FinDiff}) that the field $(t,x)\mapsto(\cV^4_t(x))/(1+|x|^2)$ is also a solution of an equation of type \eqref{eq:00}, with data $(1/(1+|\cdot|^2),0,0)$, which therefore satisfies Assumption \ref{assumption loc regularitydata}. Applying \eqref{eq:seged degen}, we then get
\begin{equation}\label{eq:02}
\E^{1/q'}\|\cV^4\|^{q'}_{L_\infty(\bB_R)}\leq N 
R^2\E^{1/q'}\|\cV^4/(1+|\cdot|^2)\|_{L_{\infty}([0,T],W^1_{p}(B_R))}^{q'}\leq NR.
\end{equation}
Hence,
$$
\E\|\cU^R\|^q_{(R)}\leq \cE + NR^2\E^{1/q}\|\cV^3\cV^5\|^{q^2}_{(R)}.
$$
Next, we can write 
$$
|\cV^5(x)|=|\eta_{\gamma_{t,x}^R}^{-1}(Y_{0,\gamma_{t,x}^R}^{-1}(Y_{\gamma_{t,x}^R,t}^{-1}(x)))|\leq\sup_{(s,y)\in\bB_R}|\eta^{-1}_s(Y_{0,s}^{-1}(y))|.
$$
By It\^o's formula, $\eta^{-1}$ is the solution of an SDE of the same type as $\eta$, in fact its differential was already given in the proof of Theorem \ref{thm:rep}, see \eqref{eq:eta tilde}. Hence, the field $(\eta^{-1}_s(Y_{0,s}^{-1}(y)))_{s\in[0,T],y\in\R^d}$ is again a solution of an equation of type \eqref{eq:00}, with the data $(\psi,f, g)$ replaced by $(1,0,0)$, and the zero order coefficients $(c,\mu)$ replaced by $(-c+|\mu|^2,-\mu)$. Hence we can estimate its supremum norm as in \eqref{eq:02}, and we can write
$$
\E\|\cU^R\|^q_{(R)}\leq \cE + NR^4\E^{1/q^2}\|\cV^3\|^{q^3}_{(R)}.
$$
Finally, $\cV^3$ can be treated similarly:
$$
|\cV^3(x)|=|U_{\gamma_{t,x}^R}(Y_{0,\gamma_{t,x}^R}^{-1}(Y_{\gamma_{t,x}^R,t}^{-1}(x)))|\leq\sup_{(s,y)\in\bB_R}|U_s(Y_{0,s}^{-1}(y))|.
$$
One can recognize the right-hand-side as $\|\cV^2\|_{\infty,\bB_R}$, which is estimated as in \eqref{eq:01}.
We can therefore conclude 
\begin{align*}
\E^{1/q}\|\cU^R\|^q_{(R)}&\leq \cE + N(2R)^{4+d/\bar p}\E^{1/q^3}\|\cV^2\|_{L_\infty(\bB_R)}^{q^3}
\\
&\leq N(2R)^{4+d/\bar p}\E^{1/q^3}\cK^{q^3}_{m,p}(\psi,f,g).
\end{align*}
Together with \eqref{eq:01} and \eqref{eq:10} this yields, for $\bar p\geq 1$, 
$$
\E\|u-{u}^R\|_{L_\infty([0,T],L_{\bar p}(B_{R-\nu R^\varepsilon})}\leq Ne^{-\delta R^{2\varepsilon}}R^{4+d}\E^{1/q^3}\cK^{q^3}_{m,p}(\psi,f,g),
$$
and since the right-hand side doesn't depend on $\bar p$, we can take the limit $\bar p\rightarrow \infty$ by Fatou's lemma. This yields \eqref{eq:loc estimate linear}, keeping in mind that $q\in(1,\infty)$ was arbitrary and that $R^{4+d}\leq Ne^{\delta'R^{2\varepsilon}}$ for any $\delta'>0$.
\qed

\bibliography{representation}{}
\bibliographystyle{Martin} 
\end{document}